\documentclass[makeidx]{article}
\title{Lower bounds for the first eigenvalue of the magnetic Laplacian\\}
\author{Bruno Colbois and Alessandro Savo}
\date{\today}
\usepackage{makeidx}
\usepackage{amssymb , amsmath, amsthm}
\usepackage{graphicx}
\newtheorem{defi}{Definition} 
\newtheorem{thm}[defi]{Theorem}
\newtheorem{ex}[defi]{Example}

 \newtheorem{prop}[defi]{Proposition}
\newtheorem{lemme}[defi]{Lemma}

\newcommand{\twosystem}[2]{\left\{\begin{aligned} &#1\\ &#2\end{aligned}\right.}
\newcommand{\threesystem}[3]{\left\{ \begin{aligned}&#1\\ &#2\\&#3\end{aligned}\right.}

\newcommand{\nero}{\smallskip$\bullet\quad$\rm}

\newcommand{\parte}[1]{\smallskip\noindent {\rm#1)}\,\,}

\newcommand{\due}[2]{#1&#2\\}

\newcommand{\matrice}{\begin{pmatrix}}
\newcommand{\ok}{\end{pmatrix}}
\newcommand{\twomatrix}[4]{\matrice\due {#1}{#2}\due{#3}{#4}\ok}

\newcommand{\derive}[2]{\dfrac{\bd #1}{\bd#2}}

\newcommand{\opd}[1]{\dfrac{\bd}{\bd #1}}

\newcommand{\Due}[2]{\begin{pmatrix}#1\\#2\end{pmatrix}}

\newcommand{\scal}[2]{\langle{#1},{#2}\rangle}

\newcommand{\abs}[1]{\lvert{#1}\rvert}

\newcommand{\reals}{{\bf R}}
\newcommand{\sphere}[1]{{\bf S}^{#1}}
\newcommand{\real}[1]{{\bf R}^{#1}}

\newcommand{\bd}{\partial}

\renewcommand{\l}{\lambda}

\begin{document}

\maketitle
\begin{abstract} 
 We consider a Riemannian cylinder $\Omega$ endowed with a closed potential $1$-form $A$ and study the magnetic Laplacian $\Delta_A$ with magnetic Neumann boundary conditions associated with those data. We establish a sharp lower bound for the first eigenvalue and show that the equality characterizes the situation where the metric is a product. We then look at the case of  a planar domain bounded by two closed curves and obtain an explicit lower bound in terms of the geometry of the domain. We finally discuss sharpness of this last estimate. 

\medskip

\noindent \it 2000 Mathematics Subject Classification. \rm 58J50, 35P15.

\noindent\it Key words and phrases. \rm Magnetic Laplacian, Eigenvalues, Upper and lower bounds, Zero magnetic field

\end{abstract}
\large

\section{Introduction} \label{intro}  Let $(\Omega,g)$ be a compact Riemannian manifold with boundary. Consider the trivial complex line bundle $\Omega\times\bf C$ over $\Omega$; its space of sections  can be identified with $C^{\infty}(\Omega,\bf C)$, the space of smooth complex valued functions on $\Omega$.  Given a smooth real 1-form $A$ on $\Omega$ we define a connection $\nabla^A$ on $C^{\infty}(\Omega,\bf C)$ as follows:
\begin{equation} \label{connection}
\nabla^A_Xu=\nabla_Xu-iA(X)u
\end{equation}
for all vector fields $X$ on $\Omega$ and for all $u\in C^{\infty}(\Omega,\bf C)$; here $\nabla$ is the Levi-Civita connection assocated to the metric $g$ of $\Omega$. The operator
\begin{equation} \label{magnetic laplacian}
\Delta_A=(\nabla^A)^{\star}\nabla ^A
\end{equation}
is called the {\it magnetic Laplacian} associated to the magnetic potential $A$, and the smooth two form 
$$
B=dA
$$
is the  associated {\it  magnetic field}. We will consider Neumann magnetic conditions, that is:
\begin{equation}\label{mneumann}
\nabla^A_Nu=0\quad\text{on}\quad\bd\Omega,
\end{equation}
where $N$ denotes the inner unit normal. 
Then, it is well-known that $\Delta_A$ is self-adjoint, and admits a discrete spectrum
$$
0\le \l_1(\Delta_A)\le \l_2(\Delta_A) \le ... \to \infty.
$$

The above  is a particular case of a more general situation, where $E\to M$ is a complex line bundle with a hermitian connection $\nabla^E$, and where the magnetic Laplacian is defined as $\Delta_E=(\nabla^E)^{\star}\nabla ^E$.

The spectrum of the magnetic Laplacian is very much studied in analysis (see for example \cite{BDP} and the references therein) and in relation with physics. For \emph{Dirichlet boundary conditions}, lower estimates of its fundamental tone have been worked out, in particular, when $\Omega$ is a planar domain and $B$ is the constant magnetic field; that is, when the function $\star B$ is constant on $\Omega$ (see for example a Faber-Krahn type inequality in \cite{Er1} and the recent\cite{LS} and the references therein, also for Neumann boundary condition). The case when the potential $A$ is a closed $1$-form is particularly interesting from  the physical point of view (Aharonov-Bohm  effect), and also from the geometric point of view. For Dirichlet boundary conditions, there is a serie of papers for domains with a pole, when the pole approaches the boundary (see \cite{AFNN, NNT} and the references therein). Last but not least, there is a Aharonov-Bohm approach to the question of  nodal and minimal partitions, see chapter 8 of \cite{BH}.

For \emph{Neumann boundary conditions}, we refer in particular to the paper \cite{HHHO}, where the authors study the multiplicity and the nodal sets corresponding to the ground state $\l_1$ for non-simply connected planar domains with harmonic potential (see the discussion below). 

Let  us also mention the recent article \cite{LLPP} (chapter 7) where the authors establish a \emph{Cheeger type inequality} for $\l_1$; that is, they find a lower bound for $\l_1(\Delta_A)$ in terms of the geometry of $\Omega$ and the potential $A$. In the preprint \cite{ELMP}, the authors approach the problem via the Bochner method.

Finally, in a more general context (see \cite{BBC}) the authors establish a lower bound for $\l_1(\Delta_A)$ in terms of the \emph{holonomy} of the vector bundle on which $\Delta_A$ acts. In both cases, implicitly, the flux of the potential $A$ plays a crucial role.

\nero From now on we will denote by $\lambda_1(\Omega,A)$ the first eigenvalue of $\Delta_A$ on $(\Omega,g)$.

 
\smallskip\subsection{Main lower bound}
Our lower bound is partly inspired by the results in \cite{HHHO} for plane domains.  First, recall that if $c$ is a closed parametrized curve (a loop), the quantity:
$$
\Phi^A_c=\dfrac{1}{2\pi}\oint_{c}A
$$
is called the {\it flux} of $A$ across $c$. (We assume that $c$ is travelled once, and we will not  specify the orientation of the loop, so that the flux will only be defined up to sign: this will not affect any of the statements, definitions or results which we will prove in this paper). Let then $\Omega$ be a fixed plane domain with one hole, and let $\Phi^A$ be the flux of the harmonic potential $A$ across the inner boundary curve. 
In Theorem 1.1 of \cite{HHHO} it is first remarked that $\lambda_1(\Omega, A)$ is positive if and only if $\Phi^A$ is not an integer
(but see the precise statement in Section 2.1 below).  Then,  it is shown that $\lambda_1(\Omega,A)$ is maximal precisely when $\Phi^A$ is congruent to $\frac 12$ modulo integers. 
The proof  relies on a delicate argument involving the nodal line of a first eigenfunction; in particular, the conclusion does not follow from a specific comparison argument, or from an explicit lower bound. 

\smallskip

In this paper we give a geometric lower bound of $\lambda_1(\Omega,A)$ when $\Omega$ is, more generally, a 
 {\it Riemannian cylinder}, that is,  a  domain $(\Omega,g)$ diffeomorphic to
$[0,1]\times\sphere 1$ endowed with a Riemannian metric $g$, and when $A$ is a closed potential $1$-form : hence, the magnetic field $B$ associated to $A$ is equal to $0$. The lower bound will depend on the geometry of $\Omega$ and,  in an explicit way,  on the flux of the potential $A$.

\medskip

Let us write $\partial\Omega=\Sigma_1\cup\Sigma_2$ where
$$
\Sigma_1=\{0\}\times\sphere 1, \quad \Sigma_2=\{1\}\times\sphere 1.
$$
We will need to foliate the cylinder by the (regular) level curves of a smooth function $\psi$ and then we introduce the following family of functions.

$$
\begin{aligned}{\cal F}_{\Omega}=\{\psi:\Omega\to\reals: \quad &\text{\it $\psi$ is constant on each boundary component}\\
&\text{\it and has no critical points inside $\Omega$}\}
\end{aligned}
$$

\medskip

As 
$\Omega$ is a cylinder, we see that ${\cal F}_{\Omega}$ is not empty. 
If $\psi\in{\cal F}_{\Omega}$, we set:
$$
K=K_{\Omega,\psi}=\dfrac{\sup_{\Omega}\abs{\nabla\psi}}{\inf_{\Omega}\abs{\nabla\psi}}.
$$
It is clear that, in the definition of the constant $K$, we can assume that the range of $\psi$ is the interval $[0,1]$, and that $\psi=0$ on $\Sigma_1$ and $\psi=1$ on $\Sigma_2$. 
Note that  the level curves of the function $\psi$ are all smooth, closed and connected; moreover they are all homotopic to each other so that the flux of a closed $1$-form $A$ across any of them is the same, and will be denoted by $\Phi^A$. 

We say, briefly, that $\Omega$ is {\it $K$-foliated by the level curves of $\psi$.} 
We also denote by $d(\Phi^A,{\bf Z})$ the minimal distance between $\Phi^A$ and the set of integer $\bf Z$:
$$
d(\Phi^A,{\bf Z})^2=\min\Big\{(\Phi^A-k)^2: k\in\bf Z\Big\}.
$$

\medskip
Finally, we say that {\it $\Omega$ is a Riemannian product} if it is isometric to $[0,a]\times\sphere 1(R)$ for suitable positive constants  $a,R$.

\begin{thm} \label{main3} \item a) Let $(\Omega,g)$ be a Riemannian cylinder, and let $A$ be a closed $1$-form on $\Omega$. Assume that $\Omega$ is $K$-foliated by the level curves of the smooth function $\psi\in{\cal F}_{\Omega}$. Then:
\begin{equation}\label{cylinder}
\lambda_1(\Omega,A)\geq\dfrac{4\pi^2}{KL^2}\cdot d(\Phi^A,{\bf Z})^2,
\end{equation}
where $L$ is the maximum length of a level curve of $\psi$ and $\Phi^A$ is the flux of $A$ across any of the boundary components of $\Omega$.

\item b) Equality holds if and only if the cylinder $\Omega$  is  a Riemannian product.
\end{thm}

\nero It is clear that we can also state the lower bound as follows:
$$
\lambda_1(\Omega,A)\geq\dfrac{4\pi^2}{\tilde K_{\Omega}}\cdot d(\Phi^A,{\bf Z})^2,
$$
where $\tilde K_{\Omega}$ is an invariant depending only $\Omega$:
$$
\tilde K_{\Omega}=\inf_{\psi\in{\cal F}_{\Omega}}K_{\Omega,\psi}L_{\psi}^2\quad\text{and}\quad L_{\psi}=\sup_{r\in {\rm range}(\psi)}\abs{\psi^{-1}(r)}.
$$
It is is not always easy to estimate $K$. In Section \ref{estimate K} we will show how to estimate $K$   in terms of the metric tensor. Note that $K\geq 1$; we will see that in many interesting situations (for example, for revolution cylinders, or for smooth embedded tubes around a closed curve) one has in fact $K=1$. 

\subsection{Doubly connected planar domains}

We now  estimate the constant $K$ above when $\Omega$ is an annular region in the plane, bounded by the inner curve $\Sigma_1$ and the outer curve $\Sigma_2$. 

\nero We assume that the inner curve $\Sigma_1$ is convex.

\medskip

From each point $x\in\Sigma_1$, consider the ray $\gamma_x(t)=x+tN_x$, where $N_x$ is the exterior normal to $\Sigma_1$ at $x$ and $t\geq 0$.  
Let $Q(x)$ be the first intersection of $\gamma_x(t)$ with $\Sigma_2$, and let
$$
r(x)=d(x,Q(x)).
$$
We say that $\Omega$ is {\it starlike with respect to $\Sigma_1$} if the map $x\to Q(x)$ is a bijection between $\Sigma_1$ and $\Sigma_2$; equivalently, if given any point $y\in\Sigma_2$, the geodesic segment which minimizes distance from $y$ to $\Sigma_1$ is entirely contained in  $\Omega$. 

For $x\in\Sigma_1$, we denote by $\theta_x$ the angle between $\gamma'_x$ and the outer normal  to $\Sigma_2$ at the point $Q(x)$, and we let
$$
m\doteq\min_{x\in\Sigma_1}{\cos\theta_x}.
$$
Note that as $\Omega$ is starlike w.r.t. $\Sigma_1$, one has $\theta_x\in [0,\frac{\pi}2]$ and then $m\geq 0$. 

\nero To have a positive lower bound, we will assume that $m>0$
(that is, $\Omega$ is {\it strictly} starlike w.r.t. $\Sigma_1$). 

\medskip

We also define
\begin{equation}\label{annulus}
\twosystem
{\beta=\min\{r(x): x\in\Sigma_1\}}
{B=\max\{r(x): x\in\Sigma_1\}}
\end{equation}

We then have the following result.

\begin{thm} \label{main2} Let $\Omega$ be an annulus in $\real 2$, which is strictly-starlike with respect to its inner (convex) boundary component $\Sigma_1$.   Assume that $A$ is a closed potential having flux $\Phi^A$ around $\Sigma_1$. Then:
$$
\lambda_1(\Omega,A)\geq \dfrac{4\pi^2}{L^2} \dfrac{\beta m}{B} d(\Phi^A,{\bf Z})^2
$$
where $\beta$ and $B$ are as in \eqref{annulus}, and $L$ is the length of the outer boundary component. If $\Sigma_2$ is also convex, then $m\geq \beta/B$ and the lower bound takes the form:
$$
\lambda_1(\Omega,A)\geq \dfrac{4\pi^2}{L^2} \dfrac{\beta^2}{B^2} d(\Phi^A,{\bf Z})^2.
$$
\end{thm}
In section \ref{sharpness}, we will explain why we need to control $\dfrac{\beta}{B}$, $L$, and why we need to impose the starlike condition.
If $\beta=B$ and $\Sigma_2$ is the circle of length $L$ we get the estimate
$$
\lambda_1(\Omega,A)\geq \dfrac{4\pi^2}{L^2} d(\Phi^A,{\bf Z})^2
$$
which is the first eigenvalue of the magnetic Laplacian on the circle with potential $A$ (see section \ref{riemannian circle}). 
If $\Sigma_2$ and $\Sigma_1$ are two concentric circles of respective lengths $L$ and $L_{\epsilon} \to L$, the domain is a thin annulus with $\lambda_1 \to \dfrac{4\pi^2}{L^2} d(\Phi^A,{\bf Z})^2$ which shows that our estimate is sharp.

\medskip
Our aim is to use these estimates on cylinders as a basis stone in order to study the same type of questions on compact surfaces of higher genus.


\section{Proof of the main theorem}

\subsection{Preliminary facts and notation} \label{preliminary} First, we recall the variational definition of the spectrum. Let $\Omega$ be a compact manifold with boundary and $\Delta_A$ the magnetic Laplacian with Neumann boundary conditions.  One verifies that
$$
\int_{\Omega}(\Delta_Au)\bar u=\int_{\Omega}\abs{\nabla^Au}^2,
$$
and the associated quadratic form is then
$$
Q_A(u)=\int_{\Omega}\abs{\nabla^Au}^2.
$$
The usual variational characterization gives:

\begin{equation}
\lambda_1(\Omega,A)= \min\Big\{ \frac{Q_A(u)}{\Vert u\Vert^2}:\ u\in C^{1}(\Omega,\mathbb C) / \{0\}\Big\}
\end{equation}

\medskip
The following proposition (which is well-known) expresses the {\it gauge invariance} of the spectrum of the magnetic Laplacian.
\begin{prop}  \label{basic facts}
\parte a The spectrum of $\Delta_A$ is equal to the spectrum of $\Delta_{A+d\phi}$ for all smooth real valued functions $\phi$; in particular, when $A$ is exact, the spectrum of $\Delta_A$ reduces to that of the classical Laplace-Beltrami operator acting on functions (with Neumann boundary conditions if $\bd\Omega$ is not empty).

\parte b If $A$ is a closed $1$-form, then $A$ is gauge equivalent to a unique (harmonic) $1$-form $\tilde A$ satisfying
$$
\twosystem
{d\tilde A=\delta\tilde A=0\quad\text{on}\quad \Omega}
{\tilde A(N)=0\quad\text{on}\quad \bd\Omega}
$$
The form $\tilde A$ is often called the {\rm Coulomb gauge} of $A$. Note that $\tilde A$ is the harmonic representative of $A$ for the absolute boundary conditions. 
\end{prop}

\begin{proof} 
\parte a  This comes from the fact that
$
\Delta_A e^{-i\phi}=e^{-i\phi} \Delta_{A+d\phi}
$
hence $\Delta_A$ and $\Delta_{A+d\phi}$ are unitarily equivalent.

\parte b Consider a solution $\phi$ of the problem:
$$
\twosystem{\Delta\phi=\delta A \quad\text{on}\quad \Omega,}
{\derive{\phi}{N}=A(N) \quad\text{on}\quad \bd \Omega.}
$$
Then one checks that $\tilde A=A-d\phi$ is a Coulomb gauge of $A$. As $\phi$ is unique up to an additive constant, $d\phi$, hence $\tilde A$, is unique. 
\end{proof}


We now focus on the first eigenvalue. 
Clearly, if $A=0$,  then $\lambda_1(\Omega,A)=0$ simply because $\Delta_A$ reduces to the usual Laplacian, which has first eigenvalue equal to zero and first eigenspace spanned by the constant functions. If $A$ is exact, then $\Delta_{A}$ is unitarily equivalent to $\Delta$, hence, again, $\lambda_1(\Omega,A)=0$. In fact one checks easily from the definition of the connection  that, if $A=d\phi$ for some real-valued function $\phi$ then
$
\nabla^{A}e^{i\phi}=0,
$
which means that $u=e^{i\phi}$ is $\nabla^A$-parallel hence $\Delta_A$-harmonic. 
On the other hand, if the magnetic field $B=dA$ is non-zero then $\lambda_1(\Omega,A)>0$.  

\smallskip
It then remains to examine the case when $A$ is closed but not exact. 
The situation was clarified in \cite{Sh} for closed manifolds and in \cite{HHHO} for Neumann boundary conditions. 

\begin{thm}\label{shikegawa}The following statements are equivalent:

\item a)
$\lambda_1(\Omega,A)=0$;

\item b)
$dA=0$ and $\Phi^A_c\in\bf Z$ for any closed curve $c$ in $\Omega$.
\end{thm}

Thus, the first eigenvalue vanishes if and only if $A$ is a closed form whose flux around every closed curve is an integer; equivalently, if $A$ has non-integral flux around at least one closed loop, then $\lambda_1(\Omega,A)>0$.


\subsection{Proof of the lower bound}  From now on we assume that $\Omega$ is a Riemannian cylinder. Fix a first eigenfunction $u$ associated to $\lambda_1(\Omega, A)$ and fix a level curve
$$
\Sigma_r=\{\psi=r\}, \quad\text{where $r\in [0,1]$.}
$$
As $\psi$ has no critical points, $\Sigma_r$ is isometric to $\sphere 1(\frac{L_r}{2\pi})$, where $L_r$ is the length of $\Sigma_r$.  The restriction of $A$ to $\Sigma_r$ is a closed $1$-form denoted by $\tilde A$; we use the restriction of $u$ to $\Sigma_r$ as a test-function for the first eigenvalue $\lambda_1(\Sigma_r,\tilde A)$ and obtain:
\begin{equation}\label{level}
\lambda_1(\Sigma_r,\tilde A)\int_{\Sigma_r}\abs{u}^2\leq\int_{\Sigma_r}\abs{\nabla^{\tilde A}u}^2.
\end{equation}
By the  estimate on the eigenvalues of a circle done in Section \ref{sectioncircle} below we see :
$$
\lambda_1(\Sigma_r,\tilde A)=\dfrac{4\pi^2}{L_r^2}d(\Phi^{\tilde A},{\bf Z})^2,
$$
where $\Phi^{\tilde A}$ is the flux of $\tilde A$ across $\Sigma_r$. Now note that $\Phi^{\tilde A}=\Phi^{A}$, because $\tilde A$ is the restriction of $A$ to $\Sigma_r$; moreover $L_r\leq L$ 
by the definition of $L$. Therefore:
\begin{equation}\label{llower}
\lambda_1(\Sigma_r,\tilde A)\geq \dfrac{4\pi^2}{L^2}d(\Phi^{ A},{\bf Z})^2
\end{equation}
for all $r$. Let $X$ be a unit vector tangent to $\Sigma_r$. Then:
$$
\begin{aligned}
\nabla^{\tilde A}_{X}u&=\nabla_{X}u-i\tilde A(X)u\\
&=\nabla_{X}u-iA(X)u\\
&=\nabla^A_{X}u.
\end{aligned}
$$
The consequence is that:
\begin{equation}\label{energy}
\abs{\nabla^{\tilde A}u}^2=\abs{\nabla^{\tilde A}_{X}u}^2=\abs{\nabla^{A}_{X}u}^2\leq \abs{\nabla^{A}u}^2.
\end{equation}
\nero {\it Note that equality holds in \eqref{energy} iff $\nabla^A_{N}u=0$ where $N$ is a unit vector normal to the level curve $\Sigma_r$ (we could take $N=\nabla\psi/\abs{\nabla\psi}$).}

\smallskip

For any fixed level curve $\Sigma_r=\{\psi=r\}$ we then have, taking into account \eqref{level}, \eqref{llower} and \eqref{energy}:
$$
\dfrac{4\pi^2}{L^2}d(\Phi^{ A},{\bf Z})^2\int_{\psi=r}\abs{u}^2\leq \int_{\psi=r}\abs{\nabla^Au}^2.
$$
Assume that $B_1\leq\abs{\nabla\psi}\leq B_2$ for positive constants $B_1,B_2$. Then the above inequality implies:
$$
\dfrac{4\pi^2}{L^2}d(\Phi^{ A},{\bf Z})^2\cdot B_1\int_{\psi=r}\dfrac{\abs{u}^2}{\abs{\nabla\psi}}\leq B_2\int_{\psi=r}\dfrac{\abs{\nabla^Au}^2}{\abs{\nabla\psi}}.
$$
We now integrate both sides from $r=0$ to $r=1$ and use the coarea formula. Conclude that
$$
\dfrac{4\pi^2}{L^2}d(\Phi^{ A},{\bf Z})^2\cdot B_1\int_{\Omega}{\abs{u}^2}\leq B_2\int_{\Omega}\abs{\nabla^Au}^2.
$$
As $u$ is a first eigenfunction, one has:
$$
\int_{\Omega}\abs{\nabla^Au}^2=\lambda_1(\Omega,A)\int_{\Omega}\abs{u}^2.
$$
Recalling that $K=\frac{B_2}{B_1}$ we finally obtain the estimate \eqref{cylinder}.  


\subsection{Proof of the equality case}\label{equalitycase}

If the cylinder $\Omega$ is a Riemannian product then it is obvious that we can take $K=1$ and then we have equality by Proposition \ref{cyl} below. Now assume that we do have equality: we have to show that $\Omega$ is a Riemannian product.  Going back to the proof, we must have the following facts.

\medskip

{\bf F1.}  {\it All level curves of $\psi$ have the same length $L$}.

\medskip

{\bf F2.}  {\it $\abs{\nabla\psi}$ must be constant and, by renormalization, we can assume that it is everywhere equal to $1$. }Then, $\psi:\Omega\to [0,a]$ for some $a>0$ and we set
$$
N\doteq\nabla\psi.
$$

{\bf F3.} {\it The eigenfunction $u$ on $\Omega$ restricts to an eigenfunction of the magnetic Laplacian of each level set $\Sigma_r=\{\psi=r\}$, with potential given by the restriction of $A$ to $\Sigma_r$.}

\medskip

{\bf F4.} {\it One has $\nabla^A_Nu=0$ identically on $\Omega$. }

\subsubsection{First step: description of the metric}

\begin{lemme} 
$\Omega$ is isometric to the product $[0,a]\times \sphere 1(\frac{L}{2\pi})$ with metric
\begin{equation}\label{metric}
g=\twomatrix 100{\theta^2(r,t)}, \quad (r,t)\in [0,a]\times [0,L]
\end{equation}
where $\theta(r,t)$ is positive and periodic of period $L$ in the variable $t$. Moreover $\theta(0,t)=1$ for all $t$.
\end{lemme}

\begin{proof} We first show that the integral curves of $N$ are geodesics; for this it is enough to show that 
$
\nabla_NN=0
$ 
on $\Omega$. Let $e_1(x)$ be a vector tangent to the level curve of $\psi$ passing through $x$. Then, we obtain a smooth vector field $e_1$ which, together with $N$, forms a global orthonormal frame.  Now 
$$
\scal{\nabla_NN}{N}=\dfrac 12 N\cdot\scal{N}{N}=0.
$$
On the other hand, as the Hessian is a symmetric tensor:
$$
\scal{\nabla_NN}{e_1}=\nabla^2\psi(N,e_1)=\nabla^2\psi(e_1,N)=\scal{\nabla_{e_1}N}{N}=\dfrac 12e_1\cdot\scal{N}{N}=0.
$$
Hence $\nabla_NN=0$ as asserted. As each integral curve of $N=\nabla\psi$ is a geodesic meeting $\Sigma_1$ orthogonally, we see that $\psi$ is actually the distance function to $\Sigma_1$. We introduce coordinates on $\Omega$ as follows. For a fixed point $p\in\Omega$ consider the unique integral curve $\gamma$ of $N$ passing through $p$ and   let $x\in\Sigma_1$ be the intersection 
of $\gamma$ with $\Sigma_1$ (note that $x$ is the foot of the unique geodesic which minimizes the distance from $p$ to $\Sigma_1$). Let $r$ be the distance of $p$ to $\Sigma_1$. We then have a map
$
\Omega\to [0,a]\times\Sigma_1
$
which sends $p$ to $(r,x)$. Its  inverse is the map $F: [0,a]\times\Sigma_1\to\Omega$ defined by
$$
F(r,x)=\exp_x(rN).
$$
Note that $F$ is a diffeomeorphism; we call the pair $(r,x)$ the {\it normal coordinates} based on $\Sigma_1$. We introduce the arc-length $t$ on $\Sigma_1$ (with origin in any assigned point of $\Sigma_1$) and recall  that  $L$ is length of $\Sigma_1$ (which is also the length of $\Sigma_2)$).
Let us compute the metric $g$ in normal coordinates. Since $N=\opd r$ one sees that
$g_{11}=1$ everywhere; for any fixed $r=r_0$ we have that $F(r_0,\cdot)$ maps $\Sigma_1$ diffeomorphically onto the level set $\{\psi=r_0\}$ so that $\opd r$ and $\opd t$ will be mapped onto orthogonal vectors, and indeed $g_{12}=0$. Setting 
$\theta(r,t)^2=\scal{\opd t}{\opd t}$ one sees that the metric takes the form \eqref{metric}.
Finally note that
$
\theta(0,t)=1
$
for all $t$, because $F(0,\cdot)$ is the identity. 
\end{proof}


\subsubsection{Second step : Gauge invariance} 

\begin{lemme} Let $\Omega$ be any Riemannian cylinder and $A=f(r,t)\,dr+h(r,t)\,dt$  a closed $1$-form on $\Omega$. Then, there exists a smooth function $\phi$ on $\Omega$ such that
$$
A+d\phi=H(t)\,dt
$$
for a smooth function $H(t)$ depending only on t. Hence, by gauge invariance, we can assume from the start that $A=H(t)\,dt$.
\end{lemme}

\begin{proof} Consider the function
$
\phi(r,t)=-\int_0^rf(x,t)\,dx.
$
Then:
$$
A+d\phi=\tilde h(r,t)\,dt
$$
for some smooth function $\tilde h(r,t)$. 
As $A$ is closed, also $A+d\phi$ is closed, which implies that $\derive{\tilde h}{r}=0$, that is,
$
\tilde h(t,r)
$
does not depend on $r$; if we set $H(t)\doteq\tilde h(t,0)$ we get the assertion.
\end{proof}

\nero We point out the following consequence. If $u=u(r,t)$ is an eigenfunction, we know from  {\bf F4} above that $\nabla^A_Nu=0$, where $N=\derive{}{r}$. As 
$
\nabla^A_Nu=\derive ur-iA(\derive{}{r})u
$
and $A=H(t)\,dt$ we obtain $A(\derive{}{r})=0$ hence
$
\derive ur=0
$
at all points of $\Omega$. This implies that 
\begin{equation}\label{uoft}
u=u(t)
\end{equation}
 depends only on $t$. 


\subsubsection{Third step :  spectrum of  circles and Riemannian products} \label{sectioncircle} In this section, we give an expression for the eigenfunctions of the magnetic Laplacian on a circle with a Riemannian metric $g$ and a closed potential $A$. Of course, we know that any metric $g$ on a  circle is always isometric to the canonical metric $g_{\rm can}=\,dt^2$, where $t$ is arc-length. But our problem in this proof is to reconstruct the global metric of the cylinder and to show that it is a product, and we cannot suppose a priori that the restricted metric of each level set of $\psi$  is the canonical metric. The same is true for the restricted potential: we know that it is Gauge equivalent to a potential of the type $a\,dt$ for a scalar $a$, but we cannot suppose a priori that it is of that form.

We refer to Appendix \ref{riemannian circle} for the complete proof of the following fact.
\begin{prop}\label{circle} Let $(M,g)$ be the circle of length $L$ endowed with the metric
$
g=\theta(t)^2\,dt^2
$
where $t\in [0,L]$ and $\theta(t)$ is a positive function, periodic of period $L$. Let $A=H(t)\,dt$. Then, the eigenvalues of the magnetic Laplacian with potential $A$ are:
$$
\lambda_k(M,A)=\dfrac{4\pi^2}{L^2}(k-\Phi^A)^2, \quad k\in\bf Z
$$
with associated eigenfunctions
$$
u_k(t)=e^{i\phi(t)}e^{\frac{2\pi i (k-\Phi^A)}{L}s(t)}, \quad k\in\bf Z.
$$
where $\phi(t)=\int_0^tH(\tau)\,d\tau$ and $s(t)=\int_0^t\theta(\tau)\,d\tau$. 

\smallskip

In particular, if the metric is the canonical one, that is, $g=dt^2$, and the potential $1$-form is harmonic, so that $A=\frac{2\pi \Phi^A}{L}dt$, then the eigenfunctions are
simply :
$$
u_k(t)=e^{\frac{2\pi i k}{L}t}, \quad k\in\bf Z.
$$
\end{prop}

We remark that if the flux $\Phi^A$ is not congruent to $1/2$ modulo integers, then the eigenvalues are all simple. If the flux is congruent to $1/2$ modulo integers, then there are two consecutive integers $k,k+1$ such that
$
\lambda_{k}=\lambda_{k+1}.
$
Consequently, the lowest eigenvalue has multiplicity two, and the first eigenspace is spanned by
$$
e^{i\phi(t)}e^{\frac{\pi i}{L}s(t)}, \, e^{i\phi(t)}e^{-\frac{\pi i}{L}s(t)}.
$$
The following proposition is an easy consequence (for a proof, see also Appendix \ref{riemannian circle}).

\begin{prop}\label{cyl} Consider the Riemannian product $\Omega=[0,a]\times\sphere 1(\frac{L}{2\pi})$, and let $A$ be a closed $1-$form on $\Omega$. Then, the spectrum of $\Delta_A$ is given by
$$
\dfrac{\pi^2 h^2}{a^2}+\dfrac{4\pi^2}{L^2}(k-\Phi^A)^2, \quad h, k\in{\bf Z}, h\geq 0.
$$
In particular,
$$
\lambda_1(\Omega,A)=\dfrac{4\pi^2}{L^2}d(\Phi^A,{\bf Z})^2.
$$
\end{prop}


\subsubsection{Fourth step : a calculus lemma} In this section, we state a technical lemma which will allow us to conclude. The proof is conceptually simple, but perhaps tricky at some points; then,  we decided to put it in Appendix \ref{technical lemma}.

\begin{lemme} \label{calculus} Let $s:[0,a]\times [0,L]\to \reals$ be a smooth, non-negative function such that
$$
s(0,t)=t,\quad s(r,0)=0, \quad s(r,L)=L \quad\text{and}\quad \derive st(r,t)\doteq\theta(r,t)>0.
$$
Assume that there exist smooth functions $p(r),q(r)$ with $p(r)^2+q(r)^2>0$ such that
$$
p(r)\cos(\frac{\pi}{L}s(r,t))+q(r)\sin(\frac{\pi}{L}s(r,t))=F(t)
$$
where $F(t)$ depends only on $t$. Then $p$ and $q$ are constant and
$
\derive sr=0
$
so that 
$$s(r,t)=t
$$
 for all $(r,t)$.
\end{lemme}


\subsubsection{End of proof of the equality case} Assume that equality holds. Then, if $u$ is an eigenfunction, we know that $u=u(t)$ by the discussion in \eqref{uoft} and $u$ restricts to an eigenfunction on each level circle $\Sigma_r$ for the potential  $A=H(t)\,dt$ above (see Fact 3 at the beginning of Section \ref{equalitycase} and the second step above). 

\medskip
We assume that $\Phi^A$ is congruent to $\frac 12$ modulo integers. This is the most difficult case; in the other cases the proof is a particular case of this, it is simpler and we omit it.


\medskip

Recall that each level set $\Sigma_r$ is a circle of length $L$ for all $r$, with metric $g=\theta(r,t)^2\,dt$. As the flux of $A$ is congruent to $\frac 12$ modulo integers, we see that there exist complex-valued functions $w_1(r),w_2(r)$ such that
$$
u(t)=e^{i\phi(t)}\Big(w_1(r)e^{\frac{\pi i}{L}  s(r,t)}+w_2(r)e^{-\frac{\pi i}{L}  s(r,t)}\Big),
$$
which, setting $f(t)=e^{-i\phi(t)}u(t)$,  we can re-write
\begin{equation}\label{rewrite}
f(t)=w_1(r)e^{\frac{\pi i}{L}  s(r,t)}+w_2(r)e^{-\frac{\pi i}{L} s(r,t)}.
\end{equation}
Recall that here $\phi(t)=\int_0^tH(\tau)\,d\tau$ and
$$
s(r,t)=\int_0^t\theta(r,\tau)\,d\tau.
$$
We take the real part on both sides of \eqref{rewrite} and obtain smooth real-valued functions $F(t), p(r),q(r)$ such that 
$$
F(t)=p(r)\cos({\frac{\pi}{L}}s(r,t))+q(r)\sin(\frac{\pi}{L} s(r,t)).
$$
Since $\theta(0,t)=1$ for all $t$, we see
$$
s(0,t)=t.
$$
Clearly $s(r,0)=0$; finally, $s(r,L)=\int_0^L\theta(r,\tau)\,d\tau=L$, being the length of the level circle $\Sigma_r$. Thus, we can apply Lemma \ref{calculus} and conclude that $s(r,t)=t$ for all $t$, that is,
$$
\theta(r,t)=1
$$
for all $(r,t)$ and the metric is a Riemannian product. 

It might happen that $p(r)=q(r)\equiv 0$. But then the real part of $f(t)$ is zero and we can work in an analogous way with the imaginary part of $f(t)$, which cannot vanish unless $u\equiv 0$. 



\subsection{General estimate of $K_{\Omega,\psi}$} \label{estimate K}

We can estimate  $K_{\Omega,\psi}$ for a Riemannian cylinder $\Omega=[0,a]\times\sphere 1$ if we know the explicit expression of the metric in the normal coordinates $(r,t)$, where $t\in [0,2\pi]$ is arc-length :
$$
g=
\left(
    \begin{array}{cc}
	g_{11} &  g_{12}  \\
	 g_{21} & g_{22}
	 \end{array}
		 \right).
$$
If $g^{ij}$ is the inverse matrix of $g_{ij}$, and if $\psi=\psi(r,t)$ one has:
$$
\abs{\nabla\psi}^2=g^{11}\Big(\derive{\psi}{r}\Big)^2+2g^{12}\derive{\psi}r\derive{\psi}{t}+
g^{22}\Big(\derive{\psi}{t}\Big)^2.
$$
The function $\psi(r,t)=r$ belongs to ${\cal F}_{\Omega}$ and one has:
$
\abs{\nabla\psi}^2=g^{11},
$
which immediately implies that we can take
$$
K_{\Omega,\psi}\leq \dfrac{\sup_{\Omega}g^{11}}{\inf_{\Omega}g^{11}}.
$$
Note in particular that if $\Omega$ is rotationally invariant, so that the metric can be put in the form:
$$
g=\left(
    \begin{array}{cc}
	1 &  0 \\
	0 & \alpha(r)^2
	 \end{array}
		 \right),
$$
for some function $\alpha(r)$,  then $K_{\Omega,\psi}=1$. The estimate becomes
\begin{equation}\label{simple}
\lambda_1(\Omega,A)\geq\dfrac{4\pi^2}{L^2}\cdot d(\Phi^A,{\bf Z})^2,
\end{equation}
where $L$ is the maximum length of a level curve $r={\rm const}$.

\begin{ex}
{\rm Yet more generally, one can fix a smooth closed curve $\gamma$ on a Riemannian surface $M$ and consider the tube of radius $R$ around $\gamma$:
$$
\Omega=\{x\in M: d(x,\gamma)\leq R\}.
$$
It is well-known that if $R$ is sufficiently small (less than the injectivity radius of the normal exponential map) then $\Omega$ is a cylinder with smooth boundary which can be foliated by the level sets of $\psi$, the distance  function to $\gamma$.
Clearly $\abs{\nabla\psi}=1$ and \eqref{simple} holds as well. 

\smallskip

A concrete example where we  could estimate the width $R$ is the case of a compact surface $M$ of genus $\ge 2$ and curvature $-a^2\le K\le -b^2$, $a \ge b >0$. Let $\gamma$ be a simple closed geodesic. Then, using the Gauss-Bonnet theorem, one can show that $R$ is bounded below by an explicit  positive constant $R=R(\gamma,a)$, hence the $R$-neighborhood of $\gamma$  is diffeomorphic to the product $S^1 \times (-1,1)$ (see for example \cite{CF}). If we take $\Omega$ as the Riemannian cylinder of width $R(\gamma,a)$ having one boundary component equal to $\gamma$  then we can foliate $\Omega$  with the  level sets of the distance function to $\gamma$ and so $K=1$ and \eqref{simple} holds,  with $L$ given by the length of the other boundary component.}
 
\end{ex}


\section{Proof of Theorem \ref{main2}: plane annuli} \label{convex}

Let $\Omega$ be an annulus in $\real 2$, which is starlike with respect to its inner convex boundary component $\Sigma_1$.   Assume that $A$ is a closed potential having flux $\Phi^A$ around $\Sigma_1$. Recall that we have to show:
\begin{equation}\label{annuliestimate}
\lambda_1(\Omega,A)\geq \dfrac{4\pi^2}{L^2} \dfrac{\beta m}{B} d(\Phi^A,{\bf Z})^2
\end{equation}
where $\beta, B$ and $m$ will be recalled  below and $L$ is the length of the outer boundary component.  If we assume that $\Sigma_2$ is also convex, then we show that 
$m\geq \beta/B$ and the lower bound takes the form:
\begin{equation}\label{annuliestimatetwo}
\lambda_1(\Omega,A)\geq \dfrac{4\pi^2}{L^2} \dfrac{\beta^2}{B^2} d(\Phi^A,{\bf Z})^2.
\end{equation}

Before giving the proof let us recall notation. For $x\in\Sigma_1$, the ray $\gamma_x$ is the geodesic segment $\gamma_x(t)=x+tN_x$, where $N_x$ is the exterior normal to $\Sigma_1$ at $x$ and $t\geq 0$.  
The  ray $\gamma_x$ meets $\Sigma_2$ at a first point $Q(x)$, and we let
$
r(x)=d(x,Q(x)).
$
For $x\in\Sigma_1$, we denote by $\theta_x$ the angle between the ray $\gamma'_x$ and the outer normal  to $\Sigma_2$ at the point $Q(x)$, and we let
$$
m\doteq\min_{x\in\Sigma_1}{\cos\theta_x}.
$$
We assume that $\Omega$ is strictly starlike, that is, $m>0$; in particular $Q(x)$ is unique. Recall also that:
\begin{equation}\label{annulus}
\beta=\min_{x\in\Sigma_1}r(x), \quad B=\max_{x\in\Sigma_1}r(x).
\end{equation}
We construct a suitable smooth function $\psi$ and estimate the constant $K=K_{\Omega,\psi}$ with respect to the geometry of $\Omega$.
The starlike assumption implies that each point in $\Omega$ belongs to a unique ray $\gamma_x$. Then we can define a function $\psi:\Omega\to [0,1]$ as follows:
$$
\psi=\threesystem
{0\quad\text{on}\quad\Sigma_1}
{1\quad\text{on}\quad\Sigma_2}
{\text{linear on each ray from $\Sigma_1$ to $\Sigma_2$}.}
$$
Estimates \eqref{annuliestimate} and \eqref{annuliestimatetwo} now follow from Theorem \ref{main3} together with the following Proposition. 
\begin{prop} \label{estimate doubly convex} 
\parte a At all points of $\Omega$ one has:
$
\frac{1}{B}\leq\abs{\nabla\psi}\leq\frac{1}{\beta m}.
$
Therefore:
$$
K_{\Omega,\psi}=\dfrac{\sup_{\Omega}\abs{\nabla\psi}}{\inf_{\Omega}\abs{\nabla\psi}}\leq\dfrac{B}{\beta m}.
$$
\parte b One has 
$$
\sup_{r\in [0,1]}\abs{\psi^{-1}(r)}=L=\abs{\Sigma_2}.
$$
\parte c If $\Sigma_2$ is also convex, then $m\geq \beta/B$ hence we can take $K=\beta^2/B^2$. 
\end{prop}

The proof of the Proposition \ref{estimate doubly convex} depends on the following steps. 

\medskip

 {\bf Step 1.} {\it On the ray $\gamma_x$ joining $x$ to $Q(x)$, consider the point $Q_t(x)$ at distance $t$ from $x$, and let $\theta_x(t)$  be the angle between $\gamma'_x$ and $\nabla\psi(Q_t(x))$. Then the function
$$
h(t)=\cos(\theta_x(t))
$$
is non-increasing in $t$. As $\theta_x(r(x))=\theta_x$ we have in particular:
$$
\cos(\theta_x(t))\geq \cos(\theta_x)\geq m
$$
for all $t\in [0,r(x)]$ and $x\in\Sigma_1$.}

\medskip

{\bf Step 2.} {\it The function $r\to\abs{\psi^{-1}(r)}$ is non-decreasing in $r$.}

\medskip

{\bf Step 3.} {\it If $\Sigma_2$ is also convex we have $m\geq \beta/B$.}

\medskip

We will prove Steps 1-3  below. 

\medskip

{\bf Proof of Proposition \ref{estimate doubly convex}}.  a) At any point of $\Omega$, let $\nabla^R\psi$ denote the radial part of $\nabla\psi$, which is the gradient of the restriction of $\psi$ to the ray passing through the given point. As such restriction is a linear function, one sees that
$$
\dfrac{1}{B}\leq\abs{\nabla^R\psi}\leq \dfrac{1}{\beta}.
$$
Since $\abs{\nabla\psi}\geq\abs{\nabla^R\psi}$ one gets immediately
$$
\abs{\nabla\psi}\geq\dfrac{1}{B}.
$$
Note that  $\theta_x(t)$, as defined above, is precisely the angle between $\nabla\psi$ and $\nabla^R\psi$, so that, using Step 1,

$$
\abs{\nabla^R\psi}=\abs{\nabla\psi}\cos\theta_x(t)\geq m\abs{\nabla\psi}
$$
hence:
$$
\abs{\nabla\psi}\leq \dfrac{1}{m}\abs{\nabla^R\psi}\leq\dfrac{1}{\beta m}.
$$
as asserted. It is clear that b) and c)  are  immediate consequences of Steps 2-3. 

\medskip

\textbf{Proof of Step 1.} We use a suitable parametrization of $\Omega$. Let $l$ be the length of $\Sigma_1$ and  consider a parametrization
$\gamma:[0,l]\to \Sigma_1$ by arc-length $s$ with origin at a given point in $\Sigma_1$. Let $N(s)$ be the outer normal vector to $\Sigma_1$ at the point $\gamma(s)$.  Consider the set:
$$
\tilde\Omega=\{(t,s)\in [0,\infty)\times [0,l): t\leq \rho(s)\}
$$
where we have set $\rho(s)=r(\gamma(s))$. 
The starlike property implies that the map 
$
\Phi:\tilde\Omega\to \Omega
$
defined by
$$
\Phi(t,s)=\gamma(s)+tN(s)
$$
is a diffeomorphism.
Let us compute the Euclidean metric tensor in the coordinates $(t,s)$. Write $\gamma'(s)=T(s)$ for the unit tangent vector to $\gamma$ and observe that  $N'(s)=k(s)T(s)$, where $k(s)$ is the curvature of $\Sigma_1$ which is everywhere non-negative because $\Sigma_1$ is convex. Then:
$$
\twosystem
{d\Phi(\dfrac{\bd}{\bd t})=N(s)}
{d\Phi(\dfrac{\bd}{\bd s})=(1+tk(s))T(s)}
$$
If we set $\Theta(t,s)=1+t k(s)$ the metric tensor is:
$$
g=\twomatrix{1}{0}{0}{\Theta^2}
$$
and an orthonormal basis is then $(e_1,e_2)$, where
$$
e_1=\dfrac{\bd}{\bd t}, \quad e_2=\dfrac{1}{\Theta}\dfrac{\bd}{\bd s}.
$$
In these coordinates, our function $\psi$ is written:
$$
\psi(t,s)=\dfrac{t}{\rho(s)}.
$$
Now
$$
\twosystem
{\scal{\nabla\psi}{e_1}=\derive{\psi}{t}=\dfrac{1}{\rho(s)}}
{\scal{\nabla\psi}{e_2}=\dfrac{1}{\Theta}\derive{\psi}{s}=-\dfrac{t\rho'(s)}{\Theta(t,s)\rho(s)^2}}.
$$
It follows that
$$
\abs{\nabla\psi}^2=\dfrac{1}{\rho^2}+\dfrac{t^2\rho'^2}{\Theta^2\rho^4}=
\dfrac{\Theta^2\rho^2+t^2\rho'^2}{\Theta^2\rho^4}.
$$
Recall the radial gradient, which is the orthogonal projection of $\nabla\psi$ on the ray, whose direction is given by $e_1$. If we fix $x\in\Sigma_1$, we have
$$
\theta_x(t)=\text{angle between $\nabla\psi$ and $e_1$}
$$
and we have to study the function
$$
h(t)=\cos\theta_x(t)=\dfrac{\scal{\nabla\psi}{e_1}}{\abs{\nabla\psi}}=\dfrac{1}{\rho(s)\abs{\nabla\psi}}
$$
for a fixed $s$.  From the above expression of $\abs{\nabla\psi}$ and a suitable manipulation we see
$$
h(t)^2=\dfrac{\Theta^2}{\Theta^2+t^2g^2}
$$
where $g=\rho'(s)/\rho(s)$. Now
$$
\begin{aligned}
\dfrac{d}{dt}\dfrac{\Theta^2}{\Theta^2+t^2g^2}&=\dfrac{2t\Theta g^2}{(\Theta^2+t^2g^2)^2}
(t\derive{\Theta}{t}-\Theta)\\
\end{aligned}
$$
As $\Theta(t,s)=1+tk(s)$ one sees that $t\derive{\Theta}{t}-\Theta=-1$ hence
$$
\dfrac{d}{dt}h(t)^2=-\dfrac{2t\Theta g^2}{(\Theta^2+t^2g^2)^2}\leq 0
$$
Hence $h(t)^2$ is non-increasing and, as $h(t)$ is positive, it is itself non-increasing. \qed

\medskip

{\bf Proof of Step 2.} In the coordinates $(t,s)$ the curve $\psi^{-1}(r)$ is parametrized by $\alpha:[0,l]\to\tilde\Omega$ as follows:
$$
\alpha(u)=(r\rho(u),u)\quad u\in [0,l].
$$
Then:
$$
\begin{aligned}
\abs{\psi^{-1}(r)}&=\int_0^l\sqrt{g(\alpha'(u),\alpha'(u))}\,du\\
&=\int_0^l\sqrt{r^2\rho'(u)^2+(1+rk(u)\rho(u))^2}\,du
\end{aligned}
$$
Convexity of $\Sigma_1$  implies that $k(u)\geq 0$ for all $u$; differentiating under the integral sign with respect to $r$  one sees that indeed $\frac{d}{dr}\abs{\psi^{-1}(r)}\geq 0$ for all $r\in [0,1]$. 

\medskip

{\bf Proof of Step 3.} Let  $T_x$ be  the tangent line to $\Sigma_2$ at $Q(x)$ and $H(x)$  the point of $T_x$ closest to $x$. As $\Sigma_2$ is convex, $H(x)$  is not an interior point of $\Omega$, hence
$$
d(x,H(x))\geq\beta.
$$
The triangle formed by $x, Q(x)$ and $H(x)$ is rectangle in $H(x)$, then we have:
$$
r(x)\cos\theta_x=d(x,H(x)).
$$
As $r(x)\leq B$ we conclude:
$$
B \cos\theta_x\geq \beta,
$$
which gives the assertion. \quad


\section{Sharpness of the lower bound} \label{sharpness}
 
\subsection{An upper bound}  In this short paragraph, we give a simple way to get an upper bound when the potential $A$ is \emph{closed}. Then, we will use this in different kinds of examples, in order to show that the assumptions of Theorem \ref{main2} are sharp. The geometric idea is the following: if we have a region $D \subset \Omega$ such that the first absolute cohomology group $H^1(D)$ is $0$, then we can estimate from above the spectrum of $\Delta_A$ in $\Omega$ in terms of the spectrum of the usual Laplacian on $D$.
The reason is that the potential $A$ is $0$ on $D$ up to a gauge transformation; then,  on $D$, $\Delta_A$ becomes the usual Laplacian  and any eigenfunction of the Laplacian on $D$ may be extended by $0$ on $\Omega$ and thus used as a test function for the magnetic Laplacian on the whole of $\Omega$.
 
\smallskip
 
Let us give the details.  Let $D$ be a closed subset of $\Omega$ such that, for some (small) $\delta>0$ one has $H^1(D^{\delta},\reals)=0$, where  $D^{\delta}=\{p\in \Omega: {\rm dist}(p,D) < \delta\}$. This happens when $D^{\delta}$ has a retraction onto $D$. We write
$$
\partial D= (\partial D\cap \partial \Omega) \cup (\partial D \cap \Omega)=\partial^{\rm ext}D\cup\partial^{\rm int}D
$$
and we denote by $(\nu_j(D))_{j=1}^{\infty}$ the spectrum of the Laplacian acting on functions, with the Neumann boundary condition on $\partial^{\rm ext}D$ (if non empty) and the Dirichlet boundary condition on $\partial^{\rm int}D$.

\begin{prop}  \label{upperharmonic} Let $\Omega$ be a compact manifold with smooth boundary and $A$  a closed potential on $\Omega$. Assume that  $D\subset \Omega$ is a compact subdomain such that $H^1(D,\textbf R)=H^1(D^{\delta},\textbf R)=0$ for some $\delta>0$. Then we have
$$
\lambda_k(\Omega,A) \le \nu_k(D)
$$
for each $k\geq 1$.
\end{prop}
 
\noindent
\textbf{Proof.} We recall that for any function $\phi$ on $\Omega$, the operator $\Delta_A$ and $\Delta_{A+d\phi}$ are unitarily equivalent and have the same spectrum. As $A$ is closed and, by assumption,  $H^1(D^{\delta},\textbf R)=0$, $A$ is exact on $D^{\delta}$ and there exists a function $\tilde \phi$ on $D^{\delta}$ such that $A+d\tilde\phi=0$ on $D^{\delta}$.
 
\smallskip
We consider the restriction of $\tilde\phi$ to $D$ and extend it differentiably on $\Omega$ by using a partition of unity $(\chi_1,\chi_2)$ subordinated to $(D^{\delta},\Omega/D)$. Then, setting
$$
\phi\doteq\chi_1 \tilde\phi
$$
we see that $\phi$ is a smooth function on $\Omega$ which is equal to $\tilde\phi$ on $D$ so that, on $D$, one has $A+d\phi=0$.
We consider the new potential $\tilde A=A+d\phi$ and observe that $\tilde A=0$ on $D$.
 
\smallskip
 
Now consider an eigenfunction $f$ for the mixed problem on $D$ (Neumann boundary conditions on $\partial^{\rm ext}D$ and  Dirichlet boundary conditions on $\partial^{\rm int}D$), and  extend it by $0$ on $\Omega\setminus D$. As $\tilde A=0$ on $D$, we see that
$$
\abs{\nabla^{\tilde A}f}^2=\abs{\nabla f}^2,
$$
and we get a test function  having the same Rayleigh quotient as that of $f$. Thanks to the usual min-max characterization of the spectrum, we obtain, for all $k$:
$$
\lambda_k(\Omega,A) = \lambda_k(\Omega,\tilde A)\le \nu_k(D).
$$

 
 \subsection{Sharpness}
We will use Proposition \ref{upperharmonic} to show the sharpness of the hypothesis in Theorem \ref{main2}. Let us first show that we need to control the ratio $\frac{BL}{\beta}$.
 
\begin{ex} \label{Blarge} \rm In the first situation, we give an example where the ratio $\frac{BL}{\beta} \to \infty$ and  the distance $\beta$ between the two components of the boundary is uniformly bounded from below. We want to show that $\lambda_1\to 0$. We consider an annulus $\Omega$ composed of two concentric balls of radius $1$ and $R+1$ and same center, with $R\to \infty$. We have $B=\beta=R$ and $L \to \infty$.

From the assumptions we get the existence of a point $x\in \Omega$ such that the ball $B(x,\frac{R}{2})$ of center $x$ and radius $\frac{R}{2}$ is contained in $\Omega$. Proposition \ref{upperharmonic} implies that $\lambda_1(\Omega,A)$ is bounded from above by the first eigenvalue of the Dirichlet problem for the Laplacian of the ball, which is proportional to $\frac{1}{R^2}$ and tends to zero because $R\to\infty$. 
 \end{ex}

\begin{ex} \label{example1} \rm Next, we  construct an example
to show that if the distance $\beta$ tends to $0$ and $B$ and $L$ are uniformly bounded from below and from above, then again $\lambda_1 \to 0$. We again use Proposition \ref{upperharmonic}.
Fix  the rectangles :
$$
R_2 = [-4,4]\times [0,4], \quad R_{1,\epsilon}=[-3,3]\times [\epsilon,2]
$$
and consider the region $\Omega_{\epsilon}$ given by the closure of $R_2\setminus R_{1,\epsilon}$. Note that $\Omega_{\epsilon}$ is a planar annulus whose boundary components are convex and get closer and closer as $\epsilon\to 0$.
 
\begin{figure}[h]\centering
\includegraphics[width=70mm]{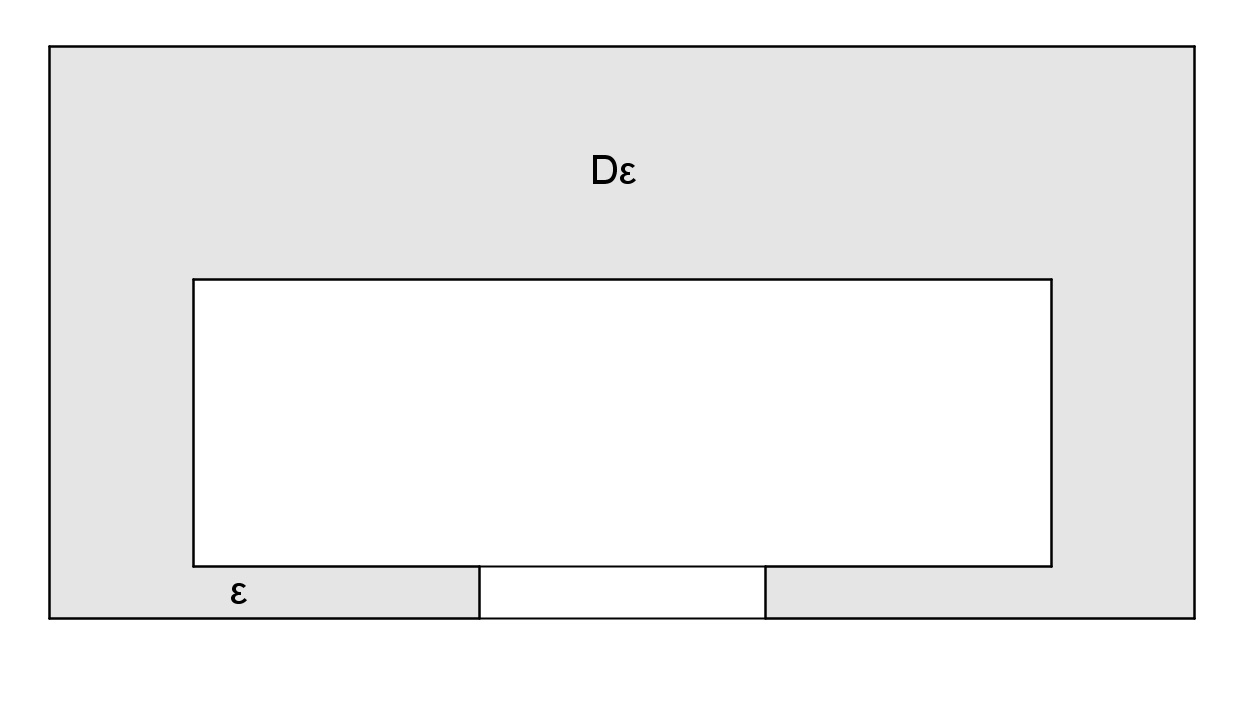}
\caption{$\lambda_1 \to 0$ as $\epsilon \to 0$}
\end{figure}
 
We show that, for any closed potential $A$ one has:
\begin{equation}\label{small}
\lim_{\epsilon\to 0}\lambda_1(\Omega_{\epsilon},A)=0.
\end{equation}
Consider the simply connected region $D_{\epsilon}\subset\Omega_{\epsilon}$ given by the complement of the rectangle $[-1,1]\times [0,\epsilon]$. Now $D_{\epsilon}$ has trivial $1$-cohomology; by Proposition \ref{upperharmonic}, to show \eqref{small} it is enough to show that
\begin{equation}\label{enough}
\lim_{\epsilon\to 0}\nu_1(D_{\epsilon})=0.
\end{equation}
By the min-max principle :
$$
\nu_1(D_{\epsilon})=\inf\Big\{ \frac{\int_{D_{\epsilon}}\vert \nabla f\vert^2}{\int_{D_{\epsilon}}f^2} : f=0\,\,\text{on}\,\, \bd D_{\epsilon}^{\rm int} \Big\}
$$
where
$$
\bd D_{\epsilon}^{\rm int} =\{(x,y)\in\Omega_{\epsilon}:x=\pm 1, y\in [0,\epsilon]\}.
$$
Define the test-function $f:D_{\epsilon}\to\reals$ as follows.
$$
f=\threesystem{1\quad\text{on the complement of $[-2,2]\times[0,\epsilon]$}}
{x-1\quad\text{on $[1,2]\times [0,\epsilon]$}}
{-x-1\quad\text{on $[-2,-1]\times [0,\epsilon]$}}
$$
One checks easily that, for all $\epsilon$:
$$
\int_{D_{\epsilon}}\vert \nabla f\vert^2=2\epsilon, \quad \int_{D_{\epsilon}}f^2\geq {\rm const}>0
$$
Then \eqref{enough} follows immediately by observing that the Rayleigh quotient of $f$ tends to $0$ as $\epsilon\to 0$
\end{ex}

\begin{ex} \label{example small} \rm In the example we constructed previously the two boundary components approach each other along a common set of positive measure (precisely, a segment of total length $6$). In the next example we sketch a construction showing that, in fact, this is not necessary. 

\smallskip

So, let us fix the outside curve $\Sigma_2$ and choose a family of inner convex curves $\Sigma_1$ such that $B$ is bounded below (say, $B\ge 1$) and $\beta \to 0$ (no other assumption is made).  Then, we want to show that $\lambda_1(\Omega,A)\to 0$. 

\smallskip

Fix points $x\in \Sigma_2$, $y\in \Sigma_1$ such that $d(x,y)=\beta$. We take $b=2\beta$ and  introduce the balls of center $x$ and radius $b$ and $\sqrt b$,  denoted by $B(x,b)$ and $B(x,\sqrt b)$, respectively. Then the set $D=\Omega\setminus (B(x,b)\cap\Omega)$ is simply connected so that,  by Proposition \ref{upperharmonic}:
$$
\lambda_1(\Omega,A)\leq \nu_1(D)
$$
and it remains to show that $\nu_1(D)\to 0$ as $b\to 0$.
 
\medskip

Introduce the function  $F(r)$ ( $r$ being the distance to $x$):
 $$
F(r)=\threesystem{1\quad\text{on the complement of $B(x,\sqrt b)$}}
{0\quad\text{on $B(x,b)$}}
{\frac{-2}{\ln b} (\ln r -\ln b)\quad\text{on $B(x,\sqrt b)-B(x,b)$}}
$$
and let $f$ be the restriction of $F$ to $D$. As $f=0$  on $\bd^{\rm int}D=\bd B(x,b)\cap\Omega$, we see that $f$ is a test function for the eigenvalue $\nu_1(D)$. A straightforward calculation shows that,  as $b\to 0$, we have
$$
\int_D \vert \nabla f\vert^2 \to 0;
$$
on the other hand, as $B \ge 1$, the volume of $D$ is uniformly bounded from below, which implies that
$$
\int_D f^2 \ge C >0.
$$
We conclude that the Rayleigh quotient of $f$ tends to $0$ as $b \to 0$, which shows the assertion.  
\end{ex}
 
\begin{ex} \label{example2} \rm The following example  shows that we need to impose some condition on the outer curve in order to get a positive lower bound as in Theorem \ref{main2}. 

\smallskip

It is an easy and classical fact that, in order to create a small eigenvalue for the Neumann problem, it is sufficient to deform a domain locally, near a boundary point, as indicated by the mushroom-shaped region shown in the figure below. Up to a gauge transformation, we can suppose that the potential $A$ is locally $0$ in a neigborhood of  the mushroom, and we have to estimate the first eigenvalue of the Laplacian with Dirichlet boundary condition at the basis of the mushroom  (which is a  segment of length $\epsilon$) and Neumann boundary condition on the remaining part of its boundary, as required by Proposition \ref{upperharmonic}.
 
\begin{figure}[h]\centering
\includegraphics[width=70mm]{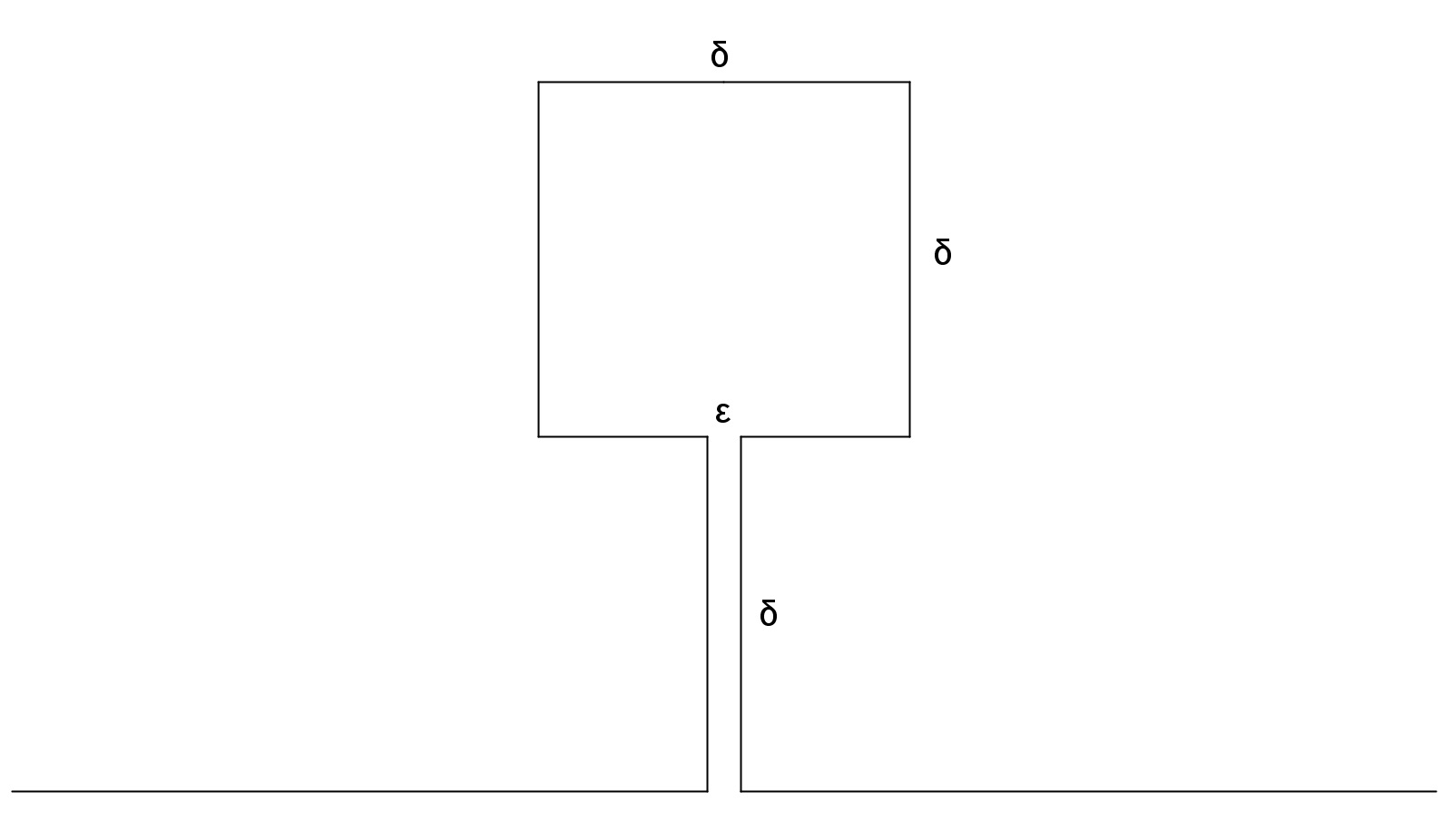}
\caption{A local deformation implying $\lambda_1 \to 0$}
\end{figure}
 
The only point is to take the value of the parameter $\epsilon$  much smaller than $\delta$ as $\delta \to 0$. Take for example $\epsilon =\delta^4$ and consider a function $u$ taking value $1$ in the square of size $\delta$ and passing linearly from $1$ to $0$ outside the rectangle of sizes $\epsilon,\delta$. The norm of the gradient of $u$ is $0$ on the square of size $\delta$ and $\frac{1}{\delta}$ in the rectangle of size $\delta,\epsilon$.

Then the Rayleigh quotient is
$$
R(u) \le \frac{\frac{1}{\delta^2}\delta \epsilon}{\delta^2}=\frac{\epsilon}{\delta^3}
$$
which tends to $0$ as $\delta \to 0$.
 
\medskip
Moreover, we can make such local deformation keeping the curvature of the boundary uniformly bounded in absolute value (see Example 2 in \cite{CGI}).
 
\end{ex}


\section{Appendix}

\subsection{Spectrum of circles and Riemannian products} \label{riemannian circle}

We first prove Proposition  \ref{circle}.

Let then $(M,g)$ be the circle of length $L$ with metric $g=\theta(t)^2dt^2$, where $t\in [0,L]$ and $\theta(t)$ is periodic of period $L$. Given the $1$-form $A=H(t)dt$ we first want to find the harmonic $1$-form $\omega$ which is cohomologous to $A$; that is, we look for  a smooth function $\phi$ so that 
$
\omega=A+d\phi
$
is harmonic. Now a unit tangent vector field to the circle is
$$
 e_1=\dfrac{1}{\theta} \dfrac{d}{dt}.
$$ 
Write  $\omega=G(t)\,dt$. Then
$$
\delta\omega=-\dfrac 1{\theta}\Big(\dfrac{G}{\theta}\Big)'.
$$
As any $1$-form on the circle is closed, we see that $\omega$ is harmonic iff $G(t)=c\theta(t)$ for a constant $c$.
We look for  $\phi$ and $c\in\reals$ so that
$$
\phi'=-H+c\theta.
$$
As $\phi$ must be periodic of period $L$, we must have $\int_0^L\phi'=0$. As the volume of $M$ is $L$, we also have $\int_0^L\theta=L$. This forces
$$
c=\dfrac{1}{L}\int_0^LH(t)\,dt.
$$
On the other hand,  as the curve $\gamma(t)=t$ parametrizes $M$ with velocity $\frac{d}{dt}$,  one sees that the flux of $A$ across $M$ is given by
$$
\Phi^A=\dfrac{1}{2\pi}\int_0^LH(t)\,dt.
$$
Therefore
$
c=\frac{2\pi}{L} \Phi^A
$
and a primitive could be
$$
\phi(t)=-\int_0^tH+c\int_0^t\theta.
$$
Conclusion:

\nero {\it The form $A=H(t)dt$ is cohomologous to the harmonic form $\omega=c\theta\, dt$ with $c=\frac{2\pi}{L} \Phi^A$}.

\medskip

We first compute the eigenvalues. By gauge invariance, we can use the potential $\omega$. In that case
$$
\Delta_{\omega}=-\nabla^{\omega}_{e_1}\nabla^{\omega}_{e_1}.
$$
Now
$$
\nabla^{\omega}_{e_1}u=\dfrac{u'}{\theta}-icu
$$
hence
$$
\nabla^{\omega}_{e_1}\nabla^{\omega}_{e_1}u=\dfrac{1}{\theta}\Big(\dfrac{u'}{\theta}-icu\Big)'-ic\Big(\dfrac{u'}{\theta}-icu\Big).
$$
After some calculation, the eigenfunction equation $\Delta_{\omega}u=\lambda u$ takes the form:
$$
-u''+\dfrac{\theta'}{\theta}u'+2ic \theta u'+c^2\theta^2 u=\lambda\theta^2 u.
$$
Recall the arc-length function $s(t)=\int_0^t\theta(\tau)\,d\tau$. We make the change of variables: 
$$
u(t)=v(s(t)), \quad\text{that is}\quad v=u\circ s^{-1}.
$$
Then:
$$
\twosystem
{u'=v'(s)\theta}
{u''=v''(s)\theta^2+v'(s)\theta'}
$$
and the equation becomes:
$$
-v''+2ic v'+c^2v=\lambda v
$$
with solutions :
$$
v_k(s)=e^{\frac{2\pi i k}{L}s}, \quad \lambda=\dfrac{4\pi^2}{L^2}(k-\Phi^A)^2, \quad k\in\bf Z.
$$
Now Gauge invariance says that 
$$
\Delta_{A+d\phi}=e^{i\phi}\Delta_Ae^{-i\phi};
$$
and  $v_k$ is an eigenfunction of $\Delta_{A+d\phi}$ iff $e^{-i\phi}v_k$ is an eigenfunction of $\Delta_A$. Hence, the eigenfunctions of $\Delta_A$ (where $A=H(t)\,dt$) are
$$
u_k=e^{-i\phi}v_k,
$$
where $\phi(t)=-\int_0^tH+c\, s(t)$ and $c=\frac{2\pi}{L}\Phi^A$. Explicitly:
\begin{equation}\label{eigenfunctions}
u_k(t)=e^{i\int_0^tH}e^{\frac{2\pi i(k-\Phi^A)s(t)}{L}}
\end{equation}
as asserted in Proposition \ref{circle}.

\smallskip

Let us know verify the last statement. If the metric is $g=dt^2$ then $\theta(t)=1$ and $s(t)=t$. If $A$ is a harmonic $1$-form then it has the expression $A=\frac{2\pi \Phi^A}{L}dt$. Taking into account \eqref{eigenfunctions} we indeed verify that $u_k(t)=e^{\frac{2\pi i k}{L}t}$.

\nero We now prove Proposition \ref{cyl}. 

\smallskip

Here we assume that $\Omega$ is a Riemannian product $[0,a]\times\sphere 1(\frac{L}{2\pi})$ with coordinates $(r,t)$ and the canonical metric on the circle. We fix a closed potential  $A$ on $\Omega$. By gauge invariance we can assume that $A$ is a Coulomb gauge, and by what we said above we have easily
$$
A=\dfrac{2\pi \Phi^A}{L}\,dt.
$$
Then $A$ restrict to zero on $[0,a]$; as $A(N)=0$ on $\bd\Omega$ the magnetic Neumann conditions reduce simply to $\derive{u}{N}=0$. At this point we apply a standard argument of separation of variables; if $\phi(r)$ is an eigenfunction of the usual Neumann Laplacian on $[0,a]$, and $v(t)$ is an eigenfunction of $\Delta_A$ on $\sphere 1(\frac{L}{2\pi})$, we see that the product $u(r,t)=\phi(r)v(t)$ is indeed an eigenfunction of $\Delta_A$ on $\Omega$. As the set of eigenfunctions we obtain that way is a complete orthonormal system in $L^2(\Omega)$, we see that each eigenvalue of the product is the sum of an eigenvalue in the Neumann spectrum of $[0,a]$ and an eigenvalue in the magnetic spectrum of the circle, as computed before. We omit further details.


\subsection{Proof of Lemma \ref{calculus}} \label{technical lemma} For simplicity of notation, we give the proof when $a=L=1$. This will not affect generality. Then, assume that $s : [0,1]\times [0,1]\to\reals$ is smooth, non-negative and satisfies 
$$
s(0,t)=t,\quad s(r,0)=0,\quad s(r,1)=1\quad\text{and}\quad \derive st(r,t)\doteq \theta(r,t)>0.
$$
Assume the identity
\begin{equation}\label{identity}
F(t)=p(r)\cos(\pi s(r,t))+q(r)\sin(\pi s(r,t))
\end{equation}
for real-valued functions $F(t),p(r),q(r)$, such that $p(r)^2+q(r)^2>0$.  Then we must show:
\begin{equation}\label{sr}
\derive sr=0
\end{equation}
everywhere. 
\medskip

Differentiate \eqref{identity} with respect to $t$ and get:
\begin{equation}\label{fprime}
F'(t)=-\pi p(r)\theta(r,t)\sin(\pi s)+\pi q(r)\theta(r,t)\cos(\pi s)
\end{equation}
and we have the following matrix identity
$$
\twomatrix{\cos(\pi s)}{\sin(\pi s)}{-\pi\theta\sin(\pi s)}{\pi\theta\cos(\pi s)}\Due{p}{q}=\Due{F}{F'}.
$$
We then see:
$$
p(r)=F(t)\cos(\pi s)-\dfrac{F'(t)}{\pi\theta}\sin(\pi s).
$$
Set $t=0$ so that $s=0$ and $p(r)=F(0)\doteq p$ is constant; the previous identity becomes 
\begin{equation}\label{p}
p=F(t)\cos(\pi s)-\dfrac{F'(t)}{\pi\theta}\sin(\pi s).
\end{equation}
Observe that:
\begin{equation}\label{changes}
\twosystem
{F'(0)=\pi q(r)\theta(r,0)}
{F'(1)=-\pi q(r)\theta(r,1)}
\end{equation}
\nero Assume $F'(0)=0$. Then, as $\theta(t,r)$ is positive one must have $q(r)=0$ for all $r$, hence $p\ne 0$ and 
$
F(t)=p\cos(\pi s),
$
from which, differentiating with respect to $r$, one gets easily
$
\derive{s}{r}=0
$
and we are finished. 

\nero We now assume that $F'(0)\ne 0$: then we see from \eqref{changes} that $q$ is not identically zero and the smooth function $F':[0,1]\to\reals$ changes sign. This implies that

\nero {\it  there exists $t_0\in (0,1)$ such that $F'(t_0)=0$.}

\smallskip

Now \eqref{p} evaluated at $t=t_0$ gives:
$$
p=F(t_0)\cos(\pi s(r,t_0))
$$
for all $r$. Differentiate w.r.t. $r$ and get, for all $r\in [0,1]$:
$$
0=\sin(\pi s(r,t_0))\derive{s}{r}(r,t_0).
$$ 
Since $s(r,t)$ is increasing in $t$, we have 
$$
0<s(r,t_0)<s(r,1)=1.
$$
Hence $\sin(\pi s(r,t_0))>0$ and we get
$$
\derive{s}{r}(r,t_0)=0.
$$
 \eqref{identity} writes:
$$
F(t)=p\cos(\pi s)+q(r)\sin(\pi s),
$$
and then, differentiating w.r.t. $r$:
$$
0=-p\pi \sin(\pi s)\derive sr+q'(r)\sin(\pi s)+\pi q(r)\cos(\pi s)\derive sr.
$$
Evaluating at $t=t_0$ we obtain
$
0=q'(r)\sin(\pi s(r,t_0))
$
which implies
$$
q'(r)=0
$$
hence $q(r)=q$, a constant. We conclude that
$$
F(t)=p\cos(\pi s)+q\sin(\pi s)
$$
for constants $p,q$.  We differentiate the above w.r.to $r$ and get:
$$
0=\Big(-\pi p \sin(\pi s)+\pi q\cos(\pi s)\Big)\derive sr
$$
for all $(r,t)\in [0,1]\times [0,1]$. Now,  the expression inside parenthesis is non-zero a.e. on the square. Then one must have $\derive sr=0$ everywhere and the final assertion follows.

\addcontentsline{toc}{chapter}{Bibliography}
\bibliographystyle{plain}
\bibliography{CS}

 \bigskip

\normalsize 
\noindent Bruno Colbois

\noindent Universit\'e de Neuch\^atel, Institut de Math\'ematiques \\
Rue Emile Argand 11\\
 CH-2000, Neuch\^atel, Suisse

\noindent bruno.colbois@unine.ch

\bigskip

\noindent Alessandro Savo

\noindent Dipartimento SBAI, Sezione di Matematica \\
Sapienza Universit\`a di Roma,  
Via Antonio Scarpa 16\\
00161 Roma, Italy

\noindent alessandro.savo@sbai.uniroma1.it

\end{document}